\documentclass[dvipsnames,onefignum,onetabnum,final]{siamart250211}

\title{Computing accurate singular vectors and eigenvectors using
  mixed-precision Jacobi algorithms}

\author{ Zhengbo~Zhou\footnote{C\MakeLowercase{orresponding author
      (\email{zhengbo.zhou@postgrad.manchester.ac.uk}).}}
  \thanks{Department of Mathematics,
    University of Manchester,
    Manchester M13 9PL,
    United~Kingdom. \\ \textbf{Funding:} The first author was supported by  the University of Manchester Research
    Scholar Award and the second author was supported by Engineering and Physical Sciences Research Council grant EP/W018101/1.} \and Fran\c{c}oise~Tisseur\footnotemark[2]
  \and Marcus~Webb\footnotemark[2] 
  }

\headers{Accurate singular vectors and eigenvectors}{
  Zhengbo~Zhou, Fran\c{c}oise~Tisseur and Marcus~Webb}

\usepackage{amsfonts}
\usepackage{amsmath}
\usepackage{amssymb} 
\usepackage{bookmark}
\usepackage{booktabs}
\usepackage{enumitem}
\usepackage{graphics,graphicx}
\usepackage{mathtools}
\usepackage[utf8]{inputenc}
\usepackage{upref}
\usepackage{subdepth}
\usepackage[normalem]{ulem}
\usepackage{comment}
\usepackage[caption=false]{subfig}
\definecolor{mycolor1}{rgb}{0.06667,0.44314,0.74510}%
\definecolor{mycolor2}{rgb}{0.86667,0.32941,0.00000}%
\definecolor{mycolor3}{rgb}{0.92941,0.69412,0.12549}%
\definecolor{mycolor4}{rgb}{0.23137,0.66667,0.19608}%
\definecolor{mycolor5}{rgb}{0.12941,0.12941,0.12941}%

\usepackage{tcolorbox}
\tcbuselibrary{listings,breakable}
\colorlet{LightCornflowerBlue}{CornflowerBlue!50}
\newtcolorbox{mybox}{
  colback=LightCornflowerBlue,
  colframe=black,
  boxrule=0pt,
  breakable,
  arc=0pt,
  outer arc=0pt
}

\makeatletter
\g@addto@macro\bfseries{\boldmath}
\makeatother

\usepackage{pgfplots,tikz}
\pgfplotsset{
  width=0.35\textwidth,
  height=0.28\textwidth,
  scale only axis,
  xlabel near ticks,
  ylabel near ticks,
  every axis title shift=3pt,
  grid=both,
}
\usetikzlibrary{plotmarks}
\pgfplotsset{compat=1.18}

\usepackage{algorithm}
\usepackage{algorithmicx}
\usepackage{algpseudocode}

\algrenewcommand\algorithmiccomment[1]{\hfill\textcolor{red}{\%\ #1}}
\usepackage{chngcntr}
\counterwithout{algorithm}{section}

\newsiamthm{assumption}{Assumption}
\newsiamremark{remark}{Remark}

\mathchardef\Gamma="7100
\mathchardef\Delta="7101
\mathchardef\Theta="7102
\mathchardef\Lambda="7103
\mathchardef\Xi="7104
\mathchardef\Pi="7105
\mathchardef\Sigma="7106
\mathchardef\Upsilon="7107
\mathchardef\Phi="7108
\mathchardef\Psi="7109
\mathchardef\Omega="710A

\newcommand{\R}{\mathbb{R}}

\newcommand{\wh}{\widehat}
\newcommand{\wt}{\widetilde}

\DeclareFontFamily{U}{mathx}{}
\DeclareFontShape{U}{mathx}{m}{n}{<-> mathx10}{}
\DeclareSymbolFont{mathx}{U}{mathx}{m}{n}
\DeclareMathAccent{\widecheck}{0}{mathx}{"71}

\newcommand{\mn}{^{m\times n}}

\newcommand{\nn}{^{n\times n}}
\newcommand{\tp}{^{T}}

\newcommand{\inv}{^{-1}}
\newcommand{\pinv}{^{\dagger}}
\newcommand{\diag}{\operatorname{diag}}

\newcommand{\sspan}{\operatorname{span}}

\newcommand{\abs}[1]{\lvert{#1}\rvert}

\newcommand{\norm}[1]{\|{#1}\|}

\newcommand{\tnorm}[1]{\norm{#1}_2}

\DeclareMathOperator{\fl}{\operatorname{f\kern.2ptl}}

\newcommand{\iter}[1]{^{(#1)}}

\newcommand{\eps}{{\varepsilon}}

\newcommand{\ul}{{u_{\ell}}}
\newcommand{\uh}{{u_{h}}}

\usepackage{listings}
\definecolor{gray}{rgb}{0.5,0.5,0.5}
\definecolor{mauve}{rgb}{0.58,0,0.82}
\definecolor{lightgrey}{rgb}{0.9,0.9,0.9}
\definecolor{darkgreen}{rgb}{0,0.6,0}
\lstdefinestyle{mystyle}{%
  language=matlab,
  morekeywords={anymatrix},
  showstringspaces=false,
  columns=flexible,
  keepspaces = true,
  basicstyle={\small\ttfamily},
  numbers=none,
  numberstyle=\tiny\color{gray},
  keywordstyle=\color{blue},
  commentstyle=\color{darkgreen},
  stringstyle=\color{mauve},
  breakatwhitespace=true,
  upquote=true,
  xleftmargin=5.0ex
}
\def\inline{\lstinline[basicstyle=\upshape\ttfamily]}
\usepackage[T1]{fontenc}
\LoadFontDefinitionFile{\encodingdefault}{cmtt}
\ExpandArgs{c}\def{\encodingdefault/cmtt/m/it}{<->ssub*cmtt/m/sl}

\def\2nsit{2--step NS iteration}

\def\At{{\wt{A}}}

\def\Qt{{\wt{Q}}}

\newcommand{\Atcomp}{{\wt{A}_{\mathrm{comp}}}}
\newcommand{\Athcomp}{{\wt{A}_{\mathrm{h,comp}}}}

\def\at{\wt{a}}

\def\gammah{\gamma_{h}}
\def\scondt{\kappa_{2}^{S}}
\def\scond{\kappa_{2}^{D}}
\def\k{\kappa_{2}}

\def\Vt{\wt{V}}
\def\Rt{\wt{R}}
\def\Rtcomp{{\wt{R}_{\mathrm{comp}}}}

\def\vh{\wh{v}}

\newcommand{\svecerrk}[1][k]{\eps_{\mathrm{sv}}\iter{#1}}
\newcommand{\evecerrk}[1][k]{\eps_{\mathrm{ev}}\iter{#1}}
\def\evgap{\operatorname{evrg}}
\def\svgap{\operatorname{svrg}}
\def\qh{\wh{q}}
\def\Qb{W}

\begin{document}
\maketitle

\begin{abstract}
Mixed-precision variants of the Jacobi algorithm for symmetric positive definite eigenproblems and the one-sided Jacobi algorithm for singular value decompositions have recently been shown to compute eigenvalues and singular values to high relative accuracy. However, these analyses do not address the accuracy of the computed eigenvectors and singular vectors. 
In this paper, we prove error bounds for the computed eigenvectors and singular vectors, where the error is measured by the sine of the angle between the vector and its computed counterpart. The obtained bounds preserve the relative gap structure of the bounds for Jacobi algorithms proved by Demmel and Veseli\'{c}, but involve the scaled condition number of the preconditioned matrix rather than that of the original matrix (the former of which is typically much smaller). 
Numerical experiments support our theoretical bounds and demonstrate that the mixed-precision preconditioned Jacobi algorithms are especially effective for ill-conditioned matrices with small absolute gaps and moderate relative gaps between eigenvalues or singular values.
\end{abstract}



\section{Introduction}
The two-sided Jacobi algorithm~\cite{jaco46} for symmetric positive definite matrices and its one-sided variant~\cite{nash75} for general matrices are iterative methods for computing the spectral decomposition and the singular value decomposition (SVD), respectively. Recently, we proposed mixed-precision variants of Jacobi algorithms that compute eigenvalues and singular values to a high relative accuracy~\cite{htwz25,ztw26}. These variants use mixed precision to compute a preconditioned matrix for which standard Jacobi algorithms are highly accurate~\cite{deve92}.
Their analyses, however, focus only on the relative accuracy of the computed eigenvalues and singular values, and do not address the accuracy of the computed eigenvectors and singular vectors.
The accuracy of these vectors is important in applications, including many-body Hamiltonian calculations in physics~\cite{yshz08} and singular subspace estimation in statistics~\cite{cazh18}, such as canonical correlation analysis~\cite{ggsw19}.

The accurate computation of eigenvectors and singular vectors by Jacobi-type algorithms has been addressed by Demmel and Veseli\'{c}~\cite{deve92}, Mathias~\cite{math95}, and Drma\v{c} and Veseli\'{c}~\cite{drve08i}.
Recent approaches for computing accurate eigenpairs~\cite{ogai18}~\cite{toio26} and singular pairs~\cite{ogai20} use iterative refinement based on Newton-type methods~\cite{dmw83}, which gradually improves existing eigenpairs and singular pairs.
Instead, we seek to prove the high accuracy of the computed eigenvectors and singular vectors by the mixed-precision Jacobi algorithms proposed in~\cite{htwz25} and~\cite{ztw26}.

To measure the accuracy of the computed right singular vectors and eigenvectors, for nonzero vectors $x$ and $y$, we use $\angle(x,y)$, 
the angle between the subspaces $\sspan\{x\}$ and $\sspan\{y\}$, defined by%
\begin{equation}
  \label{eq.subspace-diff-angle-def}
  \angle(x,y) = \arccos \bigg( \frac{|x\tp y|}{\tnorm{x}\tnorm{y}} \bigg).
\end{equation}
One can verify that $\angle(x,y) \in [0,\pi/2]$, $\angle(x,y) = \angle(y,x)$, and $\angle(Qx,Qy) = \angle(x,y)$ for any orthogonal matrix $Q$.
In this paper, we focus on%
\begin{equation}
    \label{eq.general-sin-angle-bound} 
    \sin\angle\big( v_k(A), \wh{v}_k(A) \big),
\end{equation}
where $v_k(A)$ denotes the exact right singular vector or eigenvector associated with the $k$th largest singular value or eigenvalue of $A$, respectively, and $\wh{v}_k(A)$ denotes its computed counterpart. As proved in~\cite{deve92}, the $k$th right singular vector computed by the one-sided Jacobi algorithm admits bounds of the form%
\begin{equation}
    \label{eq.deve-bound-svecs} 
    \sin\angle\big( v_k(A), \wh{v}_k(A) \big) \le p(m,n) u \scond(A)/\svgap(A,k),
\end{equation}
where $p(m,n)$ is some polynomial in $m$ and $n$, and%
\begin{align}%
    \notag
    \scond(A) &:= \k(AD), \quad D = \diag(\tnorm{a_i}\inv),\\
    \notag 
    \svgap(A,k) &:=
    \min\bigg\{2,\min_{j\neq k}\frac{\abs{\sigma_k(A)-\sigma_j(A)}}{\sigma_k(A)}\bigg\},
\end{align}
are the one-sided scaled condition number and the relative gap for the $k$th singular value, respectively. There are several different notions of a relative singular value gap in the literature; the definition of $\svgap(A,k)$ above allows us to use the results of Eisenstat and Ipsen \cite{eiip95}. The $k$th eigenvector computed by the two-sided Jacobi algorithm admits bounds of a similar form,%
\begin{equation}
    \notag 
    \sin\angle\big( v_k(A), \wh{v}_k(A) \big) \le q(n)u \scondt(A)/\evgap(A,k), \qquad 
\end{equation}
where $q(n)$ is some polynomial in $n$, and
\begin{align}
    \notag
    \scondt(A)  &:= \k(DAD),\quad D = \diag(a_{ii}^{-1/2}), \\
    \label{eq.evgap-deve}
    \evgap(A,k) &:= \min_{j\neq k}\frac{|\lambda_{k}(A)-\lambda_{j}(A)|}{\sqrt{\lambda_{k}(A)\lambda_{j}(A)}}
\end{align}
are the two-sided scaled condition number and the relative gap for the $k$th eigenvalue, respectively.
Here, $\k(A) = \sigma_{\max}(A)/\sigma_{\min}(A)$ denotes the $2$ norm condition number of $A$. 
For bidiagonalization and tridiagonalization-based methods
(e.g., QR algorithm, divide-and-conquer), 
it is the reciprocal of the absolute gaps, $\min_{j\neq k} |\sigma_k(A) - \sigma_j(A)|$ and $\min_{j\neq k} |\lambda_k(A) - \lambda_j(A)|$, that appears in the error bounds~\cite[Thm.~11.7.1]{parl98-SEP},~\cite{dges99}. 
More specifically, if $A$ has 
at least two small singular values, then their absolute gap is necessarily small, whereas their relative gap may remain moderate, leading to tighter error bounds.
The same argument applies to eigenvalues~\cite{deve92}.

Our contributions on the computed right singular vectors are as follows.
We derive an upper bound, analogous to~\eqref{eq.deve-bound-svecs}, on $\sin\angle\big( v_k(A), \wh{v}_k(A) \big)$, but with $\scond(A)$ replaced by $\scond(\At)$, where $\At = A\Vt$ is the preconditioned matrix, and $\Vt$ is the preconditioner generated by the algorithms described in~\cite[sect.~3]{ztw26}. Thus, the right singular vectors computed by our algorithm can achieve higher accuracy because $\scond(\At)$ may be moderate even when $\scond(A)$ is large, especially for ill-conditioned matrices~\cite{ztw26}.
An analogous improvement applies to the computed eigenvectors, with the corresponding scaled condition number replaced by that of the preconditioned matrix.

In this work, we consider only symmetric positive definite matrices with distinct eigenvalues and general real matrices with distinct singular values. Generalizations to multiple eigenvalues or singular values are highly nontrivial, but as a starting point we refer the interested reader to~\cite[sect.~5.5]{drve08i},~\cite{li98}, and~\cite[Chap.~V]{stsu90-MPT}. 

The numerical experiments demonstrate that our algorithm delivers computed right singular vectors with smaller errors than the other tested methods for matrices with small absolute gaps, but large relative gaps, such as ill-conditioned random matrices with geometrically distributed singular values. The same applies to the mixed-precision preconditioned two-sided Jacobi algorithm for eigenvectors.

The rest of the article is organized as follows. In section~\ref{sec.acc-svals}, we first present the mixed-precision preconditioned one-sided Jacobi algorithm and bound the error measure~\eqref{eq.def-svecerrk}. We then prove the analogous bound for the mixed-precision preconditioned two-sided Jacobi algorithm in section~\ref{sec.accuracy-of-evec}.
Numerical experiments are presented in section~\ref{sec.numerical-experiments} to support Theorems~\ref{thm.main-thm-svec} and~\ref{thm.main-thm-evec}.
Finally, conclusions are given in section~\ref{sec.conclusion}.

\section{Accuracy of the computed singular vectors} \label{sec.acc-svals}
We prove in Theorem~\ref{thm.main-thm-svec} that the mixed-precision one-sided Jacobi algorithm proposed in~\cite{ztw26} computes the right singular vectors of a general real matrix with high accuracy. 
The same accuracy result holds for the left singular vectors as a consequence of~\cite[Thm.~3.3]{eiip95} so we do not discuss it explicitly, see Remark~\ref{rmk.approach-for-left-svecs}.


\begin{algorithm}[bthp!]
\caption{Mixed-precision preconditioned one-sided Jacobi algorithm}
\label{alg.mp-precond-onesided-Jacobi}
\begin{algorithmic}[1]
\Require{A full-rank matrix $A \in \R\mn$ with $m \geq n$, two precisions $u$ and $\uh$ with $0 < \uh < u$, and a preconditioner $\Vt$ satisfying $\tnorm{\Vt\tp\Vt - I} \le p_1u < 1/2$.}
\Ensure{A computed SVD $\wh{U}\wh{\Sigma}\wh{V}\tp$ of $A$.}

\State{\label{li.compute-Atcomp}%
  Compute the preconditioned matrix $\At$ by computing the product $A\Vt$ at precision $\uh$, which gives $\Athcomp$, and then demoting it to precision $u$, which yields $\Atcomp$.}
\State{\label{li.compute-SVD}%
  Compute an SVD, $\Atcomp\tp = \wh{V}_J \wh\Sigma_{J} \wh{U}_{J}\tp$ using the one-sided Jacobi algorithm with stopping criterion~\cite[Eq.~1.1]{ztw26} at precision $u$.}
\State{\label{li.compute-right-svecs}%
  Construct the right singular vector matrix $V = \Vt \wh{V}_{J}$ at precision $u$, and set $\wh{U} = \wh{U}_{J}$ and $\wh{\Sigma} = \wh{\Sigma}_{J}$.}
\end{algorithmic}
\end{algorithm}

\begin{remark}
Algorithm~\ref{alg.mp-precond-onesided-Jacobi} is a slight modification of the mixed-precision preconditioned one-sided Jacobi algorithm~\cite{ztw26}, where
step~\ref{li.compute-SVD} 
applies the one-sided Jacobi algorithm to $\Atcomp\tp$ rather than to $\Atcomp$.
This modification is only made to simplify the analysis and does not affect the level of accuracy of the computed singular values (proved in \cite{ztw26}). 
As shown by Mathias~\cite{math95}, if $\Atcomp$ has poorly scaled columns and the one-sided Jacobi algorithm is applied on the right, then the relative forward error of the computed singular values is proportional to $u \max_k \k(\Atcomp^{(k)} D_c^{(k)})$, where $u$ is the precision at which the one-sided Jacobi algorithm is performed, $\Atcomp^{(k)}$ is the matrix after $k$ one-sided Jacobi steps with $\Atcomp^{(0)}=\Atcomp$, and $D_c^{(k)}$ is a diagonal scaling such that $\Atcomp^{(k)} D_c^{(k)}$ has unit column 2-norms.
In contrast, applying the method on the left of $\Atcomp$ (equivalently, on the right of $\Atcomp\tp$) yields a bound proportional to $u \k(\Atcomp D_c)$, where $D_c$ scales the columns of $\Atcomp$ to unit norm. Thus, the one-sided Jacobi is able to deliver high relative accuracy when the input matrix has the form $\Atcomp = BD$ or $\Atcomp = DB$ where $D$ is a diagonal scaling matrix and $B$ is well-conditioned~\cite{domo04}.
The quantities $\max_k \k(\Atcomp^{(k)} D_c^{(k)})$ and $\k(\Atcomp D_c)$ are, in general, not the same: the former reflects the worst conditioning encountered along the iteration, while the latter depends only on the initial scaling.
However, extensive numerical evidence in~\cite[sect.~7.4]{deve92} and~\cite[chap.~5]{slap92} shows that their ratio remains moderate in practice.
At present, a general theoretical explanation for this behavior is still not available; see also~\cite{drma20a}.
\end{remark}

\subsection{One-sided Jacobi and assumptions}

Let us write
\begin{equation}\label{eq.gamma-gammah}
  \gamma_{h} \coloneq \frac{nu_{h}}{1-nu_{h}} <
  \gamma \coloneq \frac{nu}{1-nu} < 1.
\end{equation}
As in~\cite{ztw26}, we make the following assumptions.
\begin{assumption}\label{ass.one-sided-Jacobi}
    Let $u$, $u_{h}$, $\gamma$, and $\gammah$ be as in \eqref{eq.gamma-gammah}, and let $A, \At\in\R^{m\times n}$ and $p_{1}$ be as in
    Algorithm~\ref{alg.mp-precond-onesided-Jacobi}.
    We assume
    \begin{enumerate}[label=\textup{(A\arabic*)}]
    \item\label{it.6nukA}
    $6nu(1-p_{1}u)\inv\k(A)<1$,
    \item\label{it.gammah<u}
    $\gammah <  \frac{(1-p_1u)u}{4(1+p_1u)\k(A)}$, and 
    \item\label{it.scond}
    $4\sqrt{m}u < 1$ and $16m\sqrt{n}u\scond(\At)< 1$.
    \end{enumerate}
    Additionally, we assume that the singular values of $A$, $\At$ and $\Atcomp$ are simple.
\end{assumption}

We focus on the following error for the $k$th computed right singular vector
\begin{equation}\label{eq.def-svecerrk}
  \svecerrk
  := \sin \angle\big(v_k(A), \vh_k(A)\big), 
\end{equation}
where $v_k(A)$ denotes the right singular vectors of $A$ associated with its $k$th largest singular value, 
and $\vh_k(A) = \fl\big(\Vt \vh_k(\Atcomp)\big)$ denotes the quantity $\Vt \vh_k(\Atcomp)$ formed at working precision $u$~\cite[Chap.~2]{high02-ASNA2}, that is, the $k$th right singular vector returned by Algorithm~\ref{alg.mp-precond-onesided-Jacobi}, where $\wh{v}_k(\Atcomp)$ denotes the computed right singular vector of $\Atcomp$ associated with its $k$th largest computed singular value. 

Our main result on the error in the right singular vectors is the following theorem. The proof is contained in section \ref{sec.boundingerrorsvecs}.
\begin{theorem}\label{thm.main-thm-svec}
Let $A,\At \in\R\mn$ be as in Algorithm~\ref{alg.mp-precond-onesided-Jacobi}. If Assumption~\ref{ass.one-sided-Jacobi} holds and $p_3 u\scond(\At)$ is bounded by $\min\{1/6, \min_{1\le k\le n} \svgap(A,k)/6\}$, where $p_3$ is defined in~\eqref{eq.def-of-p3}, then
\begin{equation} \label{eq.final-bound-for-svec}
  \svecerrk \le p_{4}u\left( 1 + \frac{\scond(\At)}{\svgap(A,k)} \right).
\end{equation}
\end{theorem}

Here, the $1$ in the parenthesis highlights the error arising from orthogonality, whereas the other term is associated with the backward error of the computed SVD. 

\begin{remark} \label{rmk.approach-for-left-svecs}
    The bound~\eqref{eq.final-bound-for-svec} also holds for the left singular vectors by letting $v_k(A)$ in~\eqref{eq.def-svecerrk} be the $k$th left singular vector associated with the $k$th largest singular value, and letting $\vh_k(A)$ be the $k$th column of $\wh{U}$ in Algorithm~\ref{alg.mp-precond-onesided-Jacobi}. 
    The analysis in this section can then be repeated by using the original version of Theorem~\ref{thm.multi-perturbation-svecs}~\cite[Thm.~3.3]{eiip95}, which does not distinguish between right and left singular vectors.
\end{remark}

\subsection{Bounding the error of the computed singular vectors}\label{sec.boundingerrorsvecs} We decompose $\svecerrk$ using Lemma~\ref{lem.sin-triangle-ineqn} to obtain,
\begin{align} \label{eq.err-evec-init-decomp}
  \svecerrk
  & \le \sin \angle\big( v_k(A), \Vt \vh_k(\Atcomp) \big) + \sin \angle\big( \Vt \vh_k(\Atcomp), \fl(\Vt \vh_k(\Atcomp)) \big) \\
  & =: E_{1} + E_{2}. \notag 
\end{align}
In the following two sections, we derive upper bounds for $E_1$ and $E_2$, and hence $\svecerrk$. 

\subsubsection{Bounding \texorpdfstring{$E_{1}$}{E1}}
We use the polar decomposition $\Vt = WH$, where $W$ is orthogonal and $H$ is symmetric positive definite~\cite[chap.~8]{high08-FM}. For brevity, let $\wh{c} = \vh_k(\Atcomp)$. Then
by Lemma~\ref{lem.sin-triangle-ineqn},
\begin{equation}\label{eq.another-decomposition}
  \sin \angle\big( v_k(A), \Vt \wh{c} \big)
  \le \sin \angle\big( v_{k}(A), W \wh{c} \big)
   + \sin \angle\big( W \wh{c}, \Vt \wh{c} \big).
\end{equation}
The second term, $\sin \angle\big( W \wh{c}, \Vt \wh{c} \big)$, can be bounded using Lemma~\ref{lem.sin-and-norm-diff},
\begin{equation}\label{eq.sin-decomp-1}
  \sin \angle( W\wh{c}, \Vt \wh{c})
  \le \frac{\tnorm{W\wh{c} - \Vt\wh{c}}}{\tnorm{W\wh{c}}}
  \le \frac{\tnorm{W - \Vt}\tnorm{\wh{c}}}{\tnorm{\wh{c}}}
  = \tnorm{W - \Vt}
  \le p_1 u.
\end{equation}
The final inequality follows from the fact that the distance between $\Vt$ and its orthogonal polar factor is bounded in terms of the loss of orthogonality of $\Vt$, namely $\tnorm{\Vt\tp \Vt - I}$~\cite[Lem.~8.17]{high08-FM}, which is, by Algorithm~\ref{alg.mp-precond-onesided-Jacobi}, bounded above by $p_1 u$.

To bound the first term in \eqref{eq.another-decomposition}, $\sin \angle\big( v_{k}(A), W \wh{c} \big)$, we require the following perturbation result, which is a slight modification of Eisenstat and Ipsen~\cite[Thm.~3.3]{eiip95}.

\begin{theorem}
\label{thm.multi-perturbation-svecs}
Let $B \in \R\mn$, and let $\Delta B\in \R\mn$ be a perturbation such that%
\begin{equation}\notag
  B + \Delta B = (I + \Delta_L) B (I + \Delta_R),
\end{equation}
where $\max\{\tnorm{\Delta_L}, \tnorm{\Delta_R}\} < 1/6$.
Then for all $k = 1,\dots,n$
\begin{equation}
  \notag
  \sin\angle\big(v_k(B), v_k(B+\Delta B)\big)
  \le
  \frac{20\sqrt{2}\,\max\{\tnorm{\Delta_L},\tnorm{\Delta_R}\}}{\svgap(B,k)},
\end{equation}
provided that
$\max\{\tnorm{\Delta_L},\tnorm{\Delta_R}\}\le \svgap(B,k)/6$.
\end{theorem}

To make use of this theorem, we need the following backward error analysis of the one-sided Jacobi algorithm.
\begin{lemma}[{\cite[Thm.~2.1]{domo04}}]
    \label{lem.bwd-err-analy-one-sided-J}
     Let $\wh{V}_{J}\wh{\Sigma}_{J}\wh{U}_{J}\tp$ be a computed SVD of $\Atcomp\tp$ using step~\ref{li.compute-SVD} of Algorithm~\ref{alg.mp-precond-onesided-Jacobi}. Then there exists $\Delta \wh{U}_{J}$, $\Delta \wh{V}_{J}$, $E_{L}$ and $E_{R}$ such that
    \begin{align}
      \label{eq.Atcomp-jacobi-svd-bwd-error}
      (I + E_{L})\Atcomp(I + E_{R})
      &= (\wh{U}_{J} + \Delta\wh{U}_{J}) \wh{\Sigma}_{J} (\wh{V}_{J} + \Delta \wh{V}_{J})\tp,\\
      \label{eq.Atcomp-jacobi-bwd-multip}
      \tnorm{E_{L}} \le p_{2}u\scond(\Atcomp)
      &,\quad\tnorm{E_{R}}, \tnorm{\Delta \wh{U}_{J}}, \tnorm{\Delta \wh{V}_{J}} \le p_{2}u,
    \end{align}
    where $\wh{U}_{J} + \Delta \wh{U}_{J}$ and $\wh{V}_{J} + \Delta \wh{V}_{J}$ are orthogonal.
\end{lemma}

Note that the $k$th column of $\wh{V}_J$ is $\wh{c}$. Theorem~\ref{thm.multi-perturbation-svecs} bounds the error between two \textit{exact} singular vectors of two different matrices.
In contrast, the term $\sin\angle \big( v_k(A), W\wh{c} \big)$ in the bound~\eqref{eq.another-decomposition} is between one exact eigenvector and a computed one.
To bridge this gap, we further decompose this term as%
\begin{equation} \label{eq.further-decompse-bridge-gap} 
    \sin\angle \big( v_k(A), W\wh{c} \big)
    \le \sin\angle \big( v_k(A), WV_J(:,k) \big)+
    \sin \angle \big(W V_J(:,k), W\wh{c} \big),
\end{equation}
where $V_J = \wh{V}_J+\Delta \wh{V}_J$. 
To bound the first term, we require the next result, which reformulates \eqref{eq.Atcomp-jacobi-svd-bwd-error} as a multiplicative perturbation of $A$ instead of $\Atcomp$.

\begin{lemma}\label{lem.computed-jacobi-mult-per-A}
With the notation of Lemma~\ref{lem.bwd-err-analy-one-sided-J} and Algorithm~\ref{alg.mp-precond-onesided-Jacobi}, we have
$$
(I + \Delta_{L}) A (I + \Delta_{R})
       = (\wh{U}_{J} + \Delta\wh{U}_{J}) \wh{\Sigma}_{J} (W V_{J})\tp,
$$
where
    \begin{align}
      \label{eq.express-DeltaL}
      \Delta_{L}
      & = E_{L}+ \Delta\At D_{\At}(\At D_{\At})\pinv + E_L\Delta\At D_{\At}(\At D_{\At})\pinv,\\
      \label{eq.express-DeltaR}
      \Delta_{R}
      & = \Vt (I + E_{R}) W\tp - I,
    \end{align}
where $\Delta\At=\Atcomp-\At$ and 
$W$ is the orthogonal polar factor of $\Vt$.
\end{lemma}
\begin{proof}
Let us write $\Atcomp$ as a multiplicative perturbation of $\At$,%
\begin{equation}
    \notag
    \Atcomp = \At + \Delta \At = \big(I + \Delta \At (\At)\pinv \big) \At.
\end{equation}
Here we have used the fact that $A$ is full rank, so that $\At$ is full rank. In addition, in order to work with scaled condition number, we have
\begin{equation}\notag
  \Atcomp = \big(I + \Delta \At (\At)\pinv \big) \At = \big(I + \Delta\At D_{\At} D_{\At}\inv \At \pinv \big) \At
   = \big( I + \Delta\At D_{\At} (\At D_{\At}) \pinv \big) \At,
\end{equation}
where $D_{\At} = \diag(\tnorm{\at_i}\inv)$, and the last equality holds since $D_{\At}$ and $\At$ of full rank yields $D_{\At}\inv \At\pinv = (\At D_{\At})\pinv$~\cite{grev66}.
Substituting the multiplicative relationship between $\Atcomp$ and $\At$ into~\eqref{eq.Atcomp-jacobi-svd-bwd-error} gives the backward error for Algorithm~\ref{alg.mp-precond-onesided-Jacobi}:%
\begin{equation}\label{eq.bwd-err-analysis-II}
  (I + E_{L})\big( I + \Delta\At D_{\At} (\At D_{\At}) \pinv \big) \At (I + E_{R})
  = (\wh{U}_{J} + \Delta\wh{U}_{J}) \wh{\Sigma}_{J} V_J\tp.
\end{equation}
Finally, multiplying \eqref{eq.bwd-err-analysis-II} on the right by $W\tp$ gives
\begin{equation}\notag
  (I + E_{L})\big( I + \Delta\At D_{\At} (\At D_{\At}) \pinv \big) A\Vt (I + E_{R}) W\tp
  = (\wh{U}_{J} + \Delta\wh{U}_{J}) \wh{\Sigma}_{J} (W V_{J})\tp,
\end{equation}
where $\Delta_L$ and $\Delta_R$ are obtained by subtracting the terms from both sides of $A$ by $I$. 
\end{proof}

In order to use Theorem~\ref{thm.multi-perturbation-svecs}, let us bound $\tnorm{\Delta_L}$ and $\tnorm{\Delta_R}$, respectively. 
Taking the norm in \eqref{eq.express-DeltaL} gives
\begin{equation}
    \notag 
    \tnorm{\Delta_L} 
    \le \tnorm{E_L} + \tnorm{\Delta\At D_{\At}}\tnorm{(\At D_{\At})\pinv} + \tnorm{E_L} \tnorm{\Delta\At D_{\At}}\tnorm{(\At D_{\At})\pinv}.
\end{equation}
The following lemma gives a bound on $\tnorm{\Delta\At D_{\At}}\tnorm{(\At D_{\At})\pinv}$.
\begin{lemma}[{\cite[sect.~2.3~\&~2.4.2]{ztw26}}]
    If Assumption~\ref{ass.one-sided-Jacobi} holds, then 
    \begin{equation}\notag 
        \tnorm{\Delta\At D_{\At}} \tnorm{(\At D_{\At})\pinv}
        \le 3\sqrt{mn} u \scond(\At),
    \end{equation}
    and $\scond(\Atcomp) \le 3\scond(\At)$.
\end{lemma}
Furthermore, Assumption~\ref{it.scond} ensures that $\tnorm{\Delta\At D_{\At}}\tnorm{(\At D_{\At})\pinv} \le 1$, and, together with the above lemma, gives 
\begin{align}
    \notag 
    \tnorm{\Delta_L}
    & \le 2\tnorm{E_L}+\tnorm{\Delta\At D_{\At}}\tnorm{(\At D_{\At})\pinv} \\
    \notag 
    & \le 2p_2u\scond(\Atcomp)+3\sqrt{mn}u\scond(\At) \\
    \notag 
    & \le (6p_2+3\sqrt{mn})u \scond(\At).
\end{align}

On the other hand, 
since $W$ is orthogonal, we have%
\begin{equation}\notag 
    \tnorm{\Delta_{R}}
    = \tnorm{\Vt(I+E_{R})W\tp-I}  
    = \tnorm{\Vt(I+E_{R}) - W}  
    \le \tnorm{\Vt-W} + \tnorm{\Vt}\tnorm{E_{R}}.
\end{equation}
Using~\eqref{eq.sin-decomp-1}, \eqref{eq.Atcomp-jacobi-bwd-multip}, and $\tnorm{\Vt} \le 1+p_1u$, which is implied by $\tnorm{\Vt\tp\Vt-I} \le p_1u$~\cite[Lem.~2.5]{ztw26}, we have%
\begin{equation}
    \notag 
    \tnorm{\Delta_{R}} \le p_1u + (1+p_1u)p_2u
    = (p_1 + p_2 + p_1p_2u)u
    \le (p_1 + p_2 + p_1p_2u)u \scond(\At).
\end{equation}
By taking
\begin{equation}
    \label{eq.def-of-p3}
    p_3 = \max\{6p_2+3\sqrt{mn}, p_1+p_2+p_1p_2u \},
\end{equation}
and using Theorem~\ref{thm.multi-perturbation-svecs} gives
\begin{equation}\label{eq.sin-decomp-3}
  \sin \angle \big( v_{k}(A), W V_{J}(:,k)\big)
  \le \frac{20\sqrt{2} p_{3} u\scond(\At)}{\svgap(A,k)}. 
\end{equation}

The second term in~\eqref{eq.further-decompse-bridge-gap} can be bounded directly using Lemma~\ref{lem.sin-and-norm-diff} and~\eqref{eq.Atcomp-jacobi-bwd-multip},
\begin{equation}\label{eq.sin-decomp-2}
  \sin \angle \big( W V_{J}(:,k), W \wh{c} \big)
  = \sin \angle \big( V_{J}(:,k), \wh{c} \big)
  \le \frac{\tnorm{\Delta \wh{V}_{J}(:,k)}}{\tnorm{V_{J}(:,k)}}
  \le p_{2}u. 
\end{equation}

Finally, combining \eqref{eq.sin-decomp-1},~\eqref{eq.sin-decomp-2} and~\eqref{eq.sin-decomp-3} gives
\begin{equation}\notag
  E_{1} \le p_{1}u + p_{2}u +
  \frac{20\sqrt{2}p_{3}u\scond(\At)}{\svgap(A,k)}. 
\end{equation}

\subsubsection{Bounding \texorpdfstring{$E_{2}$}{E2}}\label{sec.bounding-e_2}

We can bound the second term in \eqref{eq.err-evec-init-decomp} using Lemma~\ref{lem.sin-and-norm-diff},
\begin{equation}
  \notag
  \sin \angle\big( \Vt \vh_k(\Atcomp), \fl(\Vt \vh_k(\Atcomp)) \big)
  \le \frac{\tnorm{\Vt \vh_k(\Atcomp) - \fl\big(\Vt \vh_k(\Atcomp)\big)}}{\tnorm{\Vt \vh_k(\Atcomp)}}.
\end{equation}
Using the error analysis for the matrix--vector product~\cite[p.~70]{high02-ASNA2}, the numerator is bounded by $n^{1/2} \gamma \tnorm{\Vt} \tnorm{\vh_k(\Atcomp)}$. Substituting this back into the above upper bound, we have 
\begin{equation}
  \label{eq.err-evec-1}
  E_{2}
  \le n^{1/2} \gamma \frac{\tnorm{\Vt} \tnorm{\vh_k(\Atcomp)}}{\tnorm{\Vt \vh_k(\Atcomp)}}
  \le n^{1/2} \gamma \k(\Vt).
\end{equation}
Since $\tnorm{\Vt\tp\Vt-I} < 1/2$ by assumption, we have $\kappa(\Vt) \le \sqrt{3}$~\cite[Prop.~4.3]{stwu02}. 
In addition, $\gamma = nu/(1-nu) < 1$ implies $nu < 1/2$, which leads to $\gamma < 2nu$. Together, we have $E_2 \le 2\sqrt{3}n^{3/2}u$.

\section{Accuracy of the eigenvectors}
\label{sec.accuracy-of-evec}
In this section, we prove in Theorem~\ref{thm.main-thm-evec} that the two-sided Jacobi algorithm proposed in~\cite{htwz25} applied to a symmetric positive definite matrix computes eigenvectors with high accuracy.

\subsection{Two-sided Jacobi and assumptions}
The mixed-precision two-sided Jacobi algorithm for eigenvectors is the algorithm proposed in~\cite{htwz25}.

\begin{algorithm}[tbhp!]
\caption{Mixed-precision preconditioned Jacobi algorithm~\cite[Alg.~1]{htwz25}.}
\label{alg.mp-precond-Jacobi}
\begin{algorithmic}[1]
\Require{A symmetric matrix $A \in \R\nn$, two precisions $u$ and $\uh$ with $0 < \uh < u$, and a preconditioner $\Qt \in \R\nn$ such that $\tnorm{\Qt\tp\Qt - I} \le p_{1}u < 1/2$.}
\Ensure{A computed spectral decomposition $\wh{Q}\wh{\Lambda}\wh{Q}\tp$ of $A$.}
\State{Compute the preconditioned matrix $\At$ by computing the product $\Qt\tp A\Qt$ entirely at precision $\uh$, which gives $\Athcomp$, and then demoting it to precision $u$ to obtain $\Atcomp$.}
\State{Compute a spectral decomposition $\wh{Q}_J\wh{\Lambda}_{J}\wh{Q}_{J}\tp$ of $\Atcomp$ using the Jacobi algorithm with stopping criterion~\cite[Eq.~1.1]{htwz25} at precision $u$.}
\State{Construct the eigenvector matrix $\wh{Q} = \Qt \wh{Q}_J$ at precision $u$, and set $\wh{\Lambda} = \wh{\Lambda}_{J}$.}
\end{algorithmic}
\end{algorithm}

As in~\cite{htwz25}, we assume that the following assumptions hold.
\begin{assumption}\label{ass.two-sided-Jacobi}
    Let $u,\uh,\gamma$ and $\gammah$ be as in \eqref{eq.gamma-gammah}, and let 
    $A, \At \in \R\nn$ and $p_1$ be as in
    Algorithm~\ref{alg.mp-precond-Jacobi}. We assume
    \begin{enumerate}[label=\textup{(B\arabic*)}]
    \item \label{it.ass.nukappa}
    $10n^{3/2}u (1-p_1u)^{-1} \kappa_2(A) < 1$,
    \item \label{it.ass-size-gammah}
    $\gammah <  \frac{u(1-p_1u)}{16n^{1/2}\kappa_2(A)}$,
    \item \label{it.c(n,u)kappa<1/2}
    $14nu\scondt(\At) < 1$.
    \end{enumerate}
    Additionally, we assume that the eigenvalues of $A$, $\At$ and $\Atcomp$ are simple and positive.
\end{assumption}

Similar to the definition of $\svecerrk$ in equation ~\eqref{eq.def-svecerrk}, the quantity of interest is
\begin{equation}\notag
  \evecerrk := \sin \big( v_{k}(A), \vh_k(A) \big),
\end{equation}
where $v_k(A) = q_k(A)$ and $\vh_k(A) = \fl\big(\Qt \qh_{k}(\Atcomp) \big)$. Here, $q_k(A)$ denotes the eigenvector of $A$ associated with its $k$th largest eigenvalue and $\wh{q}_k(\Atcomp)$ denotes the computed eigenvector of $\Atcomp$ associated with its $k$th largest computed eigenvalue.

Our main result on the error for the computed eigenvectors is the following theorem. The proof is contained in section~\ref{sec.bound-error-comp}.

\begin{theorem}
\label{thm.main-thm-evec}
Let $A, \At \in \R\nn$ be the matrices defined in Algorithm~\ref{alg.mp-precond-Jacobi}, with $A$ positive definite. If Assumption~\ref{ass.two-sided-Jacobi}, \eqref{eq.assumption-evgap} and \eqref{eq.ass-on-eta-evgap} hold, then
\begin{equation}\notag
  \evecerrk \le p_{10}u \left(  1 + \frac{\scondt(\At)}{\evgap(A,k)} \right).
\end{equation}
\end{theorem}

\subsection{Relative gaps for eigenvalues}\label{sec.relat-gaps-eigenv}
We will first give a relationship between two different definitions of the relative gap for eigenvalues: the relative gap defined in~\eqref{eq.evgap-deve} which is used in~\cite{bade90} and \cite{deve92}, and the following definition~\cite{eiip95}
\begin{equation}\notag
  \rho(A,i) := \min_{j\neq i} \frac{|\lambda_{j}(A) - \lambda_{i}(A)|}{|\lambda_{i}(A)|}, 
\end{equation}
with the convention that $\rho(A,i) = \infty$ if $\lambda_i(A) = 0$.
\begin{lemma}\label{lem.relation-rho-evgap}
Let $A \in \R\nn$ be symmetric positive definite with simple eigenvalues, then
\begin{equation}\notag
  \frac{1}{\rho(A,i)} \le 1+\frac{1}{\evgap(A,i)}.
\end{equation}
\end{lemma}

\begin{proof}
Write $\lambda_k = \lambda_k(A)$ for any index $k$. Then by definition,
\begin{eqnarray*}
   1+\frac{1}{\evgap(A,i)} &=& \max_{j\neq i} \, 1 + \frac{\sqrt{\lambda_i\lambda_j}}{|\lambda_i-\lambda_j|} = \max_{j\neq i} \frac{\lambda_i}{|\lambda_i - \lambda_j|} \left(\left|1-\frac{\lambda_j}{\lambda_i}\right| + \sqrt{\frac{\lambda_j}{\lambda_i}} \right).
\end{eqnarray*}
One can readily confirm that $|1-x^2|+x \geq 1$ for all $x > 0$, so the bracketed term is at least 1 for any $i$ and $j$. Therefore,
\begin{equation*}
    1+\frac{1}{\evgap(A,i)} \geq \max_{j\neq i} \frac{\lambda_i}{|\lambda_i - \lambda_j|} = \frac{1}{\rho(A,i)},
\end{equation*}
as required.
\end{proof}

This lemma creates a bridge between different perturbation theories for eigenvectors~\cite{deve92}~\cite{eiip95}. The next result characterizes the change in $\evgap$ under a congruence transformation. 

\begin{lemma}\label{lem.evgapAt-evgapA}
Let $\At = \Qt\tp A\Qt$ where $A$ is symmetric positive definite with simple eigenvalues, and $\Qt$ is as in Algorithm \ref{alg.mp-precond-Jacobi}. Assume
\begin{equation}\label{eq.assumption-evgap}
  p_{1}u\bigg(1+\frac{2}{\evgap(A,i)} \bigg) \le \frac{1}{2},
\end{equation}
then
\begin{equation}\notag
  \frac{1}{\evgap(\At,i)} \le \frac{2(1+p_{1}u)}{\evgap(A,i)}.
\end{equation}
\end{lemma}

\begin{proof}
Write $\lambda_k = \lambda_k(A)$ and $\wt\lambda_k = \lambda_k(\At)$ for any index $k$. By~\cite[Thm.~5.6]{demm97-ANLA}, there exists $\theta_k$ such that $\wt\lambda_k = \lambda_k(1+\theta_k)$, where $|\theta_k| \le \tnorm{\Qt\tp\Qt-I}\le p_1 u$. Hence, for any $j \neq i$, 
\begin{equation}
    \notag 
    |\wt\lambda_i - \wt\lambda_j| \ge |\lambda_i - \lambda_j| - p_1 u (\lambda_i + \lambda_j), \quad 
    \sqrt{\wt\lambda_i \wt\lambda_j} \le (1 + p_1 u)\sqrt{\lambda_i \lambda_j}.
\end{equation}
Since $(\lambda_i + \lambda_j)/\sqrt{\lambda_i \lambda_j} \le 2 + |\lambda_i - \lambda_j|/\sqrt{\lambda_i \lambda_j}$, as squaring both sides shows that the square of the right-hand side exceeds that of the left-hand side by $4|\lambda_i-\lambda_j|/\sqrt{\lambda_i\lambda_j}\ge 0$, we have
\begin{equation}
    \notag 
    \frac{|\wt\lambda_i - \wt\lambda_j|}{\sqrt{\wt\lambda_i \wt\lambda_j}} \ge
    \frac{1}{1+p_1u} \frac{|\lambda_i - \lambda_j|}{\sqrt{\lambda_i \lambda_j}}
    \bigg( 1 - p_1 u \Big( 1+ \frac{2}{|\lambda_i - \lambda_j|/\sqrt{\lambda_i \lambda_j}} \Big) \bigg).
\end{equation}
Since $\evgap(A,i) \le |\lambda_i - \lambda_j|/\sqrt{\lambda_i \lambda_j}$ for any $j\neq i$, and the Assumption~\eqref{eq.assumption-evgap} implies the bracketed term is at least $1/2$, we have 
\begin{equation}
    \notag 
    \frac{|\wt\lambda_i - \wt\lambda_j|}{\sqrt{\wt\lambda_i \wt\lambda_j}} \ge
    \frac{\evgap(A,i)}{2(1+p_1u)}.
\end{equation}
Taking the minimum over all $j\neq i$ and taking the reciprocals completes the proof.
\end{proof}

\subsection{Bounding the error of the computed eigenvectors}
\label{sec.bound-error-comp}
Let us start with the backward error analysis of the two-sided Jacobi algorithm.
\begin{lemma}[{\cite[Thm.~3.3]{drma20a}}]
\label{lem.bwd-error-two-sided-jacobi}
Let $\wh{Q}_{J}\wh{\Lambda}_{J}\wh{Q}_{J}\tp$ be the computed spectral decomposition of the symmetric positive definite matrix $A$ using the two-sided Jacobi algorithm. Then there exists a backward error $\Delta A$ and an orthogonal matrix $Q_{J}$ such that
\begin{equation}\notag
  A + \Delta A = Q_{J} \wh{\Lambda}_{J} Q_{J}\tp, \qquad
  \tnorm{\Delta A}/\tnorm{A} \le p_{5} u, \qquad \tnorm{\wh{Q}_{J}-Q_{J}} \le p_{5}u.
\end{equation}
\end{lemma}
Let $q_{i}(A)$ denote the eigenvector of $A$ associated with its $i$th largest eigenvalue,
and let $\qh_{i}(A)$ denote the corresponding computed eigenvector, namely $\wh{Q}_{J}(:,i)$.

Similarly to the procedure in section~\ref{sec.acc-svals}, let $\Qb$ be the orthogonal polar factor of $\Qt$. We use Lemma~\ref{lem.sin-triangle-ineqn} to decompose this expression into the following four parts
\begin{align}
  \label{eq.evecerrk-decompose-E1-E4}
  \evecerrk
  & \le E_{1} + E_{2} + E_{3} + E_{4} \\
  \notag
  E_{1} = \sin \big( q_{k}(A), \Qb {Q}_{J}(:,k) \big),
  & \quad E_{2} = \sin\big(\Qb {Q}_{J}(:,k), \Qb \qh_{k}(\Atcomp)\big) \\
\notag
  E_{3} = \sin\big(\Qb \qh_{k}(\Atcomp), \Qt\qh_{k}(\Atcomp) \big),
  & \quad E_{4} = \sin\big(\Qt\qh_{k}(\Atcomp) ,\fl\big(\Qt\qh_{k}(\Atcomp) \big)\big)
\end{align}

Starting from $E_{2}$, by Lemma~\ref{lem.sin-and-norm-diff} and Lemma \ref{lem.bwd-error-two-sided-jacobi}, we have
\begin{equation}\notag
  E_{2} = \sin\big({Q}_{J}(:,k) - \qh_{k}(\Atcomp) \big)
  \le \frac{\tnorm{({Q}_{J} - \wh{Q}_{J})e_{k}}}{\tnorm{{Q}_{J}(:,k)}} \le p_{5}u. 
\end{equation}
where $e_{k}$ is the $k$th column of the identity matrix. Moreover, using the argument in~\eqref{eq.sin-decomp-1} and section~\ref{sec.bounding-e_2} gives $E_{3} \le p_{1}u$
and $E_{4} \le n^{1/2} \gamma \k(\Qt)$. Now, similarly to the analysis in section~\ref{sec.acc-svals}, the problem reduces to $\sin\big(q_{k}(A), \Qb{Q}_{J}(:,k)\big)$.

To find a relationship between $A$ and the eigenvectors of $\Atcomp$, we use $\At$ as the middleman, which will further decompose $E_{1}$ as
\begin{equation}\notag
  E_{1} \le \sin \big( q_{k}(A), \Qb q_{k}(\At) \big) +
  \sin\big(\Qb q_{k}(\At), \Qb{Q}_{J}(:,k) \big). 
\end{equation}
$\Qb q_{k}(\At)$ is the $k$th exact eigenvector of $\Qb\Qt\tp A (\Qb\Qt\tp)\tp$. Using a multiplicative perturbation result for eigenvectors~\cite[Thm.~2.2]{eiip95} and the orthogonality of $\Qb$, the first term is bounded by
\begin{align*}
    \sin\angle\big( q_{k}(A), \Qb q_{k}(\At) \big)
  & \le \frac{\tnorm{\Qb\Qt\tp\Qt\Qb\tp}\tnorm{(\Qt\Qb\tp\Qb\Qt\tp)\inv-I}}{\rho(A,k) - \tnorm{\Qb\Qt\tp\Qt\Qb\tp-I}} + \tnorm{\Qt\Qb\tp-I}\\
  & = 
  \frac{\tnorm{\Qt\tp\Qt}\tnorm{(\Qt\Qt\tp)\inv-I}}{\rho(A,k) - \tnorm{\Qt\tp\Qt-I}} + \tnorm{\Qt-\Qb}.
\end{align*}

By Lemma~\ref{lem.relation-rho-evgap} and Assumption~\eqref{eq.assumption-evgap}, we have 
\begin{equation}
    \rho(A,k) \ge \frac{\evgap(A,k)}{\evgap(A,k)+1}
    \ge \frac{\evgap(A,k)}{\evgap(A,k)+2} \ge 2p_1u \ge 2\tnorm{\Qt\tp\Qt-I},
\end{equation}
which enables us to bound the denominator below by $\rho(A,k)/2$.
For the other terms, since $\tnorm{\Qt\tp\Qt - I} \le p_1u$, we have that
$\tnorm{\Qt\tp\Qt} \le 1+p_1u$~\cite[Lem.~4.2]{stwu02} as well as $\tnorm{\Qt-\Qb} \le p_1u$~\cite[Lem.~8.17]{high08-FM}. Also, by noticing that $\tnorm{\Qt\tp \Qt - I} = \tnorm{\Qt \Qt\tp -I}$ and by using the property of Neumann series~\cite[Chap.~1, Thm.~4.20]{stew98-MA1}, we have that $\tnorm{(\Qt\Qt\tp)\inv-I} \le 2p_1u$. Substituting into the above equation yields 
\begin{equation}
    \notag 
    \sin\angle\big( q_{k}(A), \Qb q_{k}(\At) \big)
    \le \frac{2(1+p_1u)p_1u}{\rho(A,k)/2}+p_1u \le \frac{p_6u}{\rho(A,k)}
    \le p_{6}u + \frac{p_{6}u}{\evgap(A,k)},
\end{equation}
where the last inequality is by Lemma~\ref{lem.relation-rho-evgap}.

Using Lemma~\ref{lem.bwd-error-two-sided-jacobi}, after applying the Jacobi algorithm, $Q_{J}$ satisfies
\begin{equation}\notag
  \Atcomp + \Delta\Atcomp = Q_{J}\wh{\Lambda}_{J}Q_{J}\tp,
\end{equation}
where $\Delta\Atcomp$ is the backward error that arises from applying the two-sided Jacobi algorithm. Because the two-sided Jacobi algorithms are implemented with a specific stopping criterion~\cite{deve92}, $\Delta\Atcomp$ satisfies~\cite[sect.~3.4.3]{drma20a}~\cite[Thm.~3.1]{deve92}
\begin{equation}\notag
  \abs{ \big( \Delta\Atcomp \big)_{ij} }
  \le p_{7} u \sqrt{\big(\Atcomp\big)_{ii} \big(\Atcomp\big)_{jj}}.
\end{equation}

The next lemma describes the relationship between $\Atcomp$ and $\At$ 
\begin{lemma}
    [{\cite[sect.~2.3,~Eq.~2.29]{htwz25}}]
    \label{lem.Atcomp-At}
    Let $\At, \Atcomp \in \R\nn$ be the matrices defined in Algorithm~\ref{alg.mp-precond-Jacobi} and let Assumption~\ref{ass.two-sided-Jacobi} holds. 
    Then,  
    \begin{equation}
        \notag 
        \Atcomp = \At +\Delta\At,\qquad 
        |\Delta \at_{ij}| \le 2u \sqrt{\at_{ii} \at_{jj}},
    \end{equation}
    and consequently,
    \begin{equation}
        \label{eq.htwz-result-2} 
        \sqrt{\at_{ii}} \ge {\sqrt{(\Atcomp)_{ii}}}/{\sqrt{1+2u}}.
    \end{equation}
\end{lemma}

By using Lemma~\ref{lem.Atcomp-At}, the spectral decomposition $Q_{J}\wh{\Lambda}_{J}Q_{J}\tp$ can be written as an additive perturbation on $\At$,
\begin{equation}\notag
  \At + \wt{\Delta}\At = Q_{J}\wh{\Lambda}_{J}Q_{J}\tp,
  \qquad \wt\Delta\At = \Delta\At + \Delta\Atcomp,
\end{equation}
where
\begin{equation}\notag
  \frac{|(\wt\Delta\At)_{ij}|}{\sqrt{\at_{ii}\at_{jj}}}
  = \frac{\abs{\big(\Delta\Atcomp\big)_{ij}}}{\sqrt{\at_{ii}\at_{jj}}} +
  \frac{\abs{\Delta\at_{ij}}}{\sqrt{\at_{ii}\at_{jj}}},
\end{equation}
and additionally, we are able to obtain an upper bound on each term as
\begin{equation}\label{eq.wtDelta-wtA-relative-perturb}
  \frac{|(\wt\Delta\At)_{ij}|}{\sqrt{\at_{ii}\at_{jj}}}
  \le (1+2u)\frac{\abs{\big(\Delta\Atcomp\big)_{ij}}}{
  \sqrt{(\Atcomp)_{ii}(\Atcomp)_{jj}}} + 2u
  \le p_{8}u,
\end{equation}
where $p_8 \ge p_7+ 2p_7u+2$. 
Now, define $\eta := \tnorm{D_{\At} (\wt\Delta\At) D_{\At}}$, and by~\eqref{eq.wtDelta-wtA-relative-perturb}, $\eta \le np_8u$. Together with Lemma~\ref{lem.sin-and-norm-diff} and~\cite[Thm.~2.5]{deve92}, we have%
\begin{equation}\notag
  \sin\big(q_{k}(\At), {Q}_{J}(:,k) \big)
  \le \frac{\tnorm{q_k(\At)-Q_J(:,k)}}{\tnorm{q_k(\At)}}
  \le \frac{\scondt(\At)\eta}{\evgap(\At,k)} + O(\eta^{2})
  \le \frac{p_{9}u \scondt(\At)}{\evgap(A,k)},
\end{equation}
provided that 
\begin{equation}\label{eq.ass-on-eta-evgap}
    4n(1+p_1u)p_8u \le \evgap(A,k).
\end{equation}

\begin{remark}
    A possible concern in applying the eigenvector bound of Demmel and Veseli\'{c}~\cite[Thm.~2.5]{deve92} is the $O(\eta^2)$ remainder in the eigenvector expansion. As in the standard perturbation expansion obtained by differentiating the eigenvalue equation for a simple eigenvalue \cite[sect.~7.2.2]{gova13-MC4}, the remainder can be controlled by the $\eta/\evgap(\At,k)$-type bound only when $\eta$ is small relative to the eigenvalue gap $\evgap(\At,k)$. Condition~\eqref{eq.ass-on-eta-evgap}, together with $\eta \le np_8 u$ and Lemma~\ref{lem.evgapAt-evgapA}, gives
    \begin{equation}
        \notag
        \eta \le np_8 u 
        \le \frac{\evgap(A,k)}{4(1+p_1u)}
        \le \frac{1}{2}\evgap(\At,k),
    \end{equation}
    which ensures that the $O(\eta^2)$ term is controlled by the same relative-gap bound and can be absorbed into $p_9$.
\end{remark}

  

\section{Numerical Experiments}\label{sec.numerical-experiments}
\def\Qtcomp{\Qt_{\mathrm{comp}}}
\def\Rtcomp{\Rt_{\mathrm{comp}}}

We performed numerical experiments to assess the accuracy of the computed singular vectors and eigenvectors, validating Theorem \ref{thm.main-thm-svec} and Theorem \ref{thm.main-thm-evec}, and to illustrate how mixed-precision Jacobi algorithms behave relative to existing methods.
We conducted all numerical experiments in MATLAB R2025b on a MacBook Pro with an M3 Pro chip and 32 GB of RAM. We took the MATLAB implementations of Algorithms~\ref{alg.mp-precond-onesided-Jacobi} and~\ref{alg.mp-precond-Jacobi} from the repositories for our previous works~\cite{htwz25}~\cite{ztw26},\footnote{One-sided Jacobi: \url{https://github.com/zhengbo0503/Code_twz26}; Two-sided Jacobi: \url{https://github.com/zhengbo0503/Code_htwz25}.} and we provide the scripts that generated the results in this section at~\url{https://github.com/zhengbo0503/Code_twz26b}. 
To ensure reproducibility, we seeded the random number generators in all test scripts with \texttt{rng(0)}.
We simulated the quadruple precision with the Advanpix Multiprecision Computing Toolbox~\cite{mct2023} using \texttt{mp.Digits(34)}.

We tested the following algorithms for computing right singular vectors:
\begin{itemize}
    \item MP3JacobiSVD: Algorithm~\ref{alg.mp-precond-onesided-Jacobi}, where the preconditioner $\Vt$ was computed by orthogonalizing the right singular vectors, computed at precision $\ul$, from $\ul$ to $u$ using Householder QR factorization~\cite[Alg.~2]{ztw26},
    \item \texttt{DGESVJ}: LAPACK subroutine for the one-sided Jacobi algorithm,
    \item \texttt{DGEJSV}: LAPACK subroutine for the preconditioned one-sided Jacobi algorithm, and
    \item MATLAB \texttt{svd}: MATLAB built-in function for computing the SVD,
\end{itemize}
and the following algorithms for computing eigenvectors:
\begin{itemize}
    \item MP3Jacobi: Algorithm~\ref{alg.mp-precond-Jacobi}, where the preconditioner $\Qt$ was computed by orthogonalizing the eigenvectors, computed at precision $\ul$, from $\ul$ to $u$ using a Householder QR factorization~\cite[Alg.~2]{htwz25},
    \item Jacobi: implemented according to~\cite[Alg.~3.1]{deve92}, with stopping tolerance set to $n^{1/2}u$,
    \item MP2Jacobi: same as MP3Jacobi, except with $\uh=u$, and
    \item MATLAB \texttt{eig}: MATLAB built-in function for computing the spectral decomposition.
\end{itemize}

For the implementation of Algorithm~\ref{alg.mp-precond-onesided-Jacobi},
it is advisable to first compute the economy size QR factorization $\Atcomp = \Qtcomp \Rtcomp$,
where $\Qtcomp \in \R^{m\times n}$ has orthonormal columns and $\Rtcomp \in \R^{n\times n}$ is upper triangular, and then apply the one-sided Jacobi to $\Rtcomp^{\tp}$.
The reasons are that $\Atcomp^{\tp}$ is an $n\times m$ matrix with $m \ge n$ and therefore does not satisfy the input requirements of \texttt{DGESVJ}, which requires the number of rows to be at least the number of columns.
In addition, if $\Atcomp$ is tall and skinny, then a single sweep of the one-sided Jacobi applied directly to $\Atcomp^{\tp}$ would cost $O(m^{2}n)$ rather than $O(mn^{2})$.
The accuracy of the singular vectors will not degrade after applying the QR factorization, see~\cite[sect.~5.6]{drve08i}~\cite{dges99}.

Since $\svgap(A,k)$ may vary with $k$, in order to assess whether the bounds in Theorem \ref{thm.main-thm-svec} and Theorem \ref{thm.main-thm-evec} hold, we need to inspect the bounds for every single index $k$. However, in all figures except Figure \ref{fig.special-matrix-svec}, we plotted the results only for a single index $k_\star$. This index $k_\star$ was an index which maximized the error $\svecerrk$ or $\evecerrk$ by Algorithm \ref{alg.mp-precond-onesided-Jacobi} or Algorithm \ref{alg.mp-precond-Jacobi}, respectively. Nonetheless, during each numerical experiment, we checked the bounds and found that they held for all indices $k$. Note that this index $k_\star$ changed for each test matrix. The intention is to capture worst-case behavior for each given test matrix, although this is not necessarily the worst case when simultaneously comparing the new methods with the bounds and all of the standard methods.

For sections~\ref{sec.numer-exp.vary-cond} and~\ref{sec.numer-exp.vary-col}, we generated the test matrices with MATLAB's \\\inline{gallery('randsvd',[m,n],kappa,MODE)} using various sizes $m\times n$, condition numbers, and prescribed singular value distributions. The three singular value distributions, labeled $\text{MODE}=3,4,5$, correspond to geometrically distributed singular values, arithmetically distributed singular values, and random singular values whose logarithms are uniformly distributed, respectively.
Since MODE 5 generates the singular values randomly, we repeated each experiment fifteen times and reported the largest observed error.

\subsection{Varying condition number}\label{sec.numer-exp.vary-cond}

In this section, to assess the computed singular vectors, we generated the test matrix $A\in\R^{1000\times 800}$ with $\k(A)$ taking fifteen logarithmically spaced values from $10^3$ to $10^{15}$, and with three different singular value distributions.
On the other hand, to assess the computed eigenvectors, we generated $A$ in the same way, except with size $500\times 500$ and symmetric positive definiteness.

\input{figs/svec-varykappa}
\input{figs/evec-varykappa}

Figures~\ref{fig.svec-vary-kappa} and~\ref{fig.evec-vary-kappa} show the maximum errors in the computed singular vectors and eigenvectors, respectively. We first observe that our bounds in both Theorems~\ref{thm.main-thm-svec} and~\ref{thm.main-thm-evec} were valid with moderate choices of the constants, $(mn)^{1/2}$ and $3n^{1/2}$, respectively. In addition, especially for MODE $3$ and $5$, our algorithm performed best among all algorithms.

However, Figures~\ref{subfig.svec-vary-kappa.mode4} and~\ref{subfig.evec-vary-kappa.mode4} show that, for MODE 4, MATLAB \texttt{svd} and \texttt{eig} produce the smallest errors, although MP3JacobiSVD and MP3Jacobi are only marginally less accurate. The arithmetically distributed singular values in MODE 4 explain this behavior: the consecutive absolute gaps are nearly uniform and of order $1/n$, and hence the absolute gap appearing in standard perturbation bounds~\cite[p.~1206]{deve92} are not much smaller than the corresponding relative gap quantities. In contrast, for MODE $3$ and $5$, where the singular values were geometrically distributed or had uniformly distributed logarithms, the relative gaps were typically much larger than the absolute gaps, especially near the small singular values. This made the advantage of relative gap based error analysis for MP3JacobiSVD and MP3Jacobi more pronounced. Figures~\ref{subfig.svec-vary-n.mode4} and~\ref{subfig.fig.evec-vary-n.mode4} show similar behavior.

\subsection{Varying matrix size}\label{sec.numer-exp.vary-col}
In this section, to assess the computed singular vectors, we generated the test matrix $A\in\R^{1000\times n}$ with $n$ taking fifteen logarithmically spaced values from $10$ to $10^{3}$, fixed $\k(A) = 10^8$, and with three different singular value distributions. Whereas for the computed eigenvectors, we generated the test matrix in the same way, except with size $n\times n$ and symmetric positive definiteness. 

\input{figs/svec-varyn}
\input{figs/evec-varyn}

Figures~\ref{fig.svec-vary-n} and~\ref{fig.evec-vary-n} show that our algorithm performed significantly better than the other three algorithms, except for MODE 4, as expected from the discussion at the end of section~\ref{sec.numer-exp.vary-cond}.
Similarly to the previous section, our derived error bounds held with moderate choices of the error constant.

\subsection{Special matrices}\label{sec.numer-exp.special}
Previous sections used randomly generated matrices, which are not necessarily representative of real-world data. In this section, we assess our algorithms on the following two special matrices from MATLAB Gallery collection,
\begin{enumerate}[label=\upshape(\arabic*),labelwidth=!]
    \item 
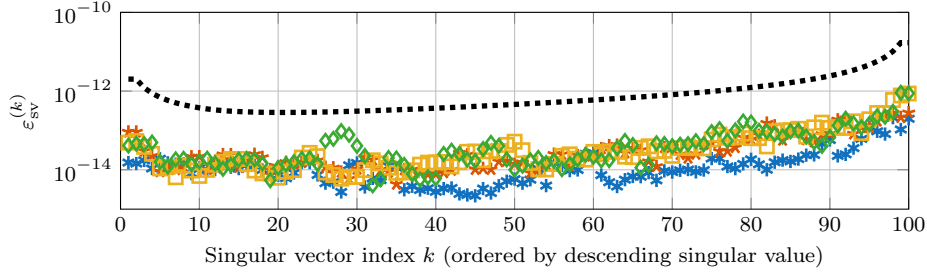
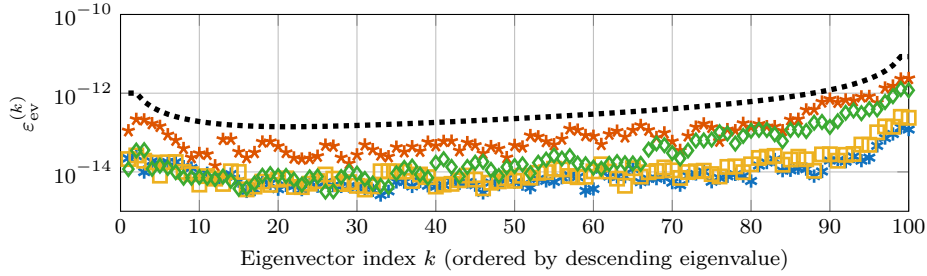
    \inline{gallery('kms',100,0.5)}: well-conditioned symmetric positive definite matrix $A\in\R^{100\times 100}$ with $\k(A) \approx 9$, and 
    \item 
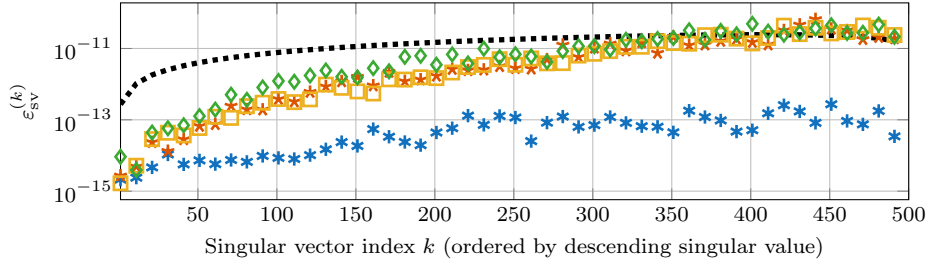
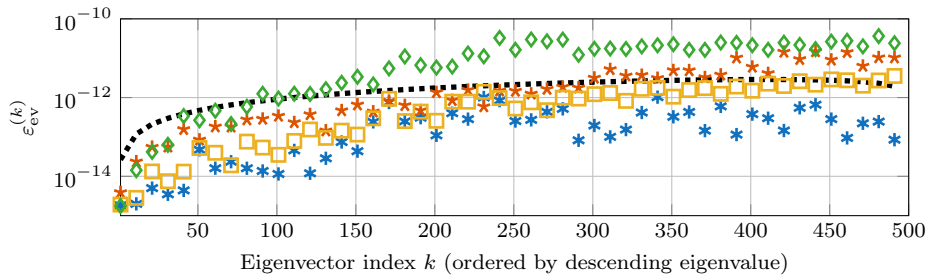
    \inline{gallery('lehmer',5e2)}: moderately ill-conditioned symmetric positive definite matrix $A\in\R^{500\times 500}$ with $\k(A) \approx 3\times 10^5$. 
\end{enumerate}
\input{figs/special-matrix-svec}

Figures~\ref{subfig.svec-specal-matrix.kms} and~\ref{subfig.evec-specal-matrix.kms} show that, when the input matrix was well-conditioned, our algorithm did not significantly outperform the other algorithms, but its error always had the same order of magnitude (or better) as that of all of the tested algorithms.
On the other hand, for ill-conditioned matrices, Figures~\ref{subfig.svec-specal-matrix.lehmer} and~\ref{subfig.evec-specal-matrix.lehmer} show that our algorithm gives smaller errors than the other algorithms, which confirms the observation in Figures~\ref{fig.svec-vary-kappa} and~\ref{fig.evec-vary-kappa}.

\subsection{Left singular vectors}
\input{figs/svec-left-varykappa}
In Remark~\ref{rmk.approach-for-left-svecs}, we mentioned that the left singular vectors also satisfy the same error bound as the right singular vectors. We verified this claim experimentally by testing our algorithms along with three other algorithms on the same test matrices as in section~\ref{sec.numer-exp.vary-cond}, and below we report the maximum error in the computed left singular vectors, $\max_k \sin \angle\big( u_k(A), \wh{u}_k(A) \big)$, in Figure~\ref{fig.svec-vary-kappa-left-svec}, where $u_k(A)$ and $\wh{u}_k(A)$ are the $k$th left singular vector of $A$ and its computed counterpart, respectively.
The results show that the errors in the computed left singular vectors also satisfied the same error bound as the right singular vectors which supports our claim in Remark~\ref{rmk.approach-for-left-svecs}.  

\section{Conclusion}\label{sec.conclusion}
We have shown that the mixed-precision Jacobi algorithms proposed in \cite{htwz25} and \cite{ztw26}, compute eigenvectors and singular vectors with high accuracy, complementing the earlier results on the relative accuracy of eigenvalues and singular values. We established error bounds for the computed vectors, measured by $\sin\angle\bigl(v_k(A), \widehat{v}_k(A)\bigr)$, that preserve the relative-gap structure of the classical bounds of Demmel and Veseli\'{c} \cite{deve92}. A key feature of our analysis is that the bounds depend on the scaled condition number of the preconditioned matrix $\widetilde{A}$ rather than that of the original matrix $A$. Since $\kappa(\widetilde{A})$ is typically much smaller than $\kappa(A)$, particularly for ill-conditioned problems, this leads to significantly improved accuracy guarantees.

The numerical experiments strongly support the theory. In particular, they demonstrate that the proposed algorithms outperform standard routines such as MATLAB \texttt{svd} and \texttt{eig}, as well as LAPACK Jacobi-based methods, when applied to ill-conditioned matrices with small absolute gaps but moderate relative gaps.

In this work, we restricted our focus to matrices with simple eigenvalues or singular values. We leave extending the analysis to quantify the angle between nontrivial subspaces for future research.

\section*{Acknowledgments}

The authors thank Yuji Nakatsukasa for suggesting the accuracy of the computed eigenvectors as an interesting topic of investigation.

\bibliographystyle{siamplain}
\bibliography{bib}

\appendix
\section{Inequalities for angles between vectors}

\begin{lemma} \label{lem.sin-triangle-ineqn}
For nonzero vectors $x,y$ and $z$, we have
\begin{equation}
  \notag
  \sin \angle(x,z) \le \sin \angle(x,y) + \sin\angle(y,z).
\end{equation}
\end{lemma}
\begin{proof}
From \cite{krei69}, we have $\angle(x,z) \le \angle(x,y) + \angle(y,z)$. Using a trigonometric inequality, we have 
\begin{align}
  \notag \sin\angle(x,z) & \le \sin \big(\angle(x,y) + \angle(y,z) \big) \\
  \notag & = \sin\angle(x,y)\cos\angle(y,z)+\sin\angle(y,z)\cos\angle(x,y) \\
  \notag & \le \sin\angle(x,y) + \sin\angle(y,z),
\end{align}
since all of the angles are between $0$ and $\pi/2$.
\end{proof} 

\begin{lemma} \label{lem.sin-and-norm-diff}
For nonzero vectors $x$ and $y$, we have $\sin\angle(x,y) \le \tnorm{x - y}/\tnorm{x}$.
\end{lemma}

\begin{proof}
Let $\mathcal{X} = \sspan\{y/\tnorm{y}\}$, and then the orthogonal projector of $\mathcal{X}$ is $P_{\mathcal{X}} = yy\tp/\tnorm{y}^2$. By the definition of angle~\eqref{eq.subspace-diff-angle-def},
$\tnorm{P_{\mathcal{X}}x } = |y\tp x|/\tnorm{y} = \tnorm{x}\cos\angle(x,y)$.
Using the Pythagorean equality~\cite[p.~10]{stsu90-MPT}, we have
\begin{equation}
    \notag 
    \tnorm{x - P_{\mathcal{X}} x}^2 = \tnorm{x}^2 - \tnorm{P_{\mathcal{X}}x}^2 
    = \tnorm{x}^2\big(1-\cos^2\angle(x,y)\big) = \tnorm{x}^2 \sin^2\angle(x,y).
\end{equation}
Consequently, we are able to bound $\sin\angle(x,y)$ as
\begin{equation}
    \notag
    \sin\angle(x,y) = \frac{\tnorm{x-P_{\mathcal{X}}x}}{\tnorm{x}}
    = \frac{\min_{z\in\mathcal{X}}\tnorm{x-z}}{\tnorm{x}}
    \le \frac{\tnorm{x-y}}{\tnorm{x}},
\end{equation}
where the second equality is due to~\cite[Thm.~2.5]{stsu90-MPT}.
\end{proof}

\section{Proof of Theorem~\ref{thm.multi-perturbation-svecs}}
\begin{proof}
This is a simplified form of the original result, which states that
\begin{equation} \label{eq.original-result-sin-angle}
  \sin \angle\big(v_k(B), v_k(B+\Delta B)\big)
  \le
  \sqrt{2}\bigg(\frac{\delta}{\svgap(B,k)-\alpha}+\beta\bigg),
\end{equation}
provided that $\alpha < \svgap(B,k)$, where
\begin{align*}
  \alpha :=\;&
               \max\Big\{
               \tnorm{(I+\Delta_L)(I+\Delta_L)\tp-I},
               \tnorm{(I+\Delta_R)\tp(I+\Delta_R)-I}
               \Big\} \\
  =\;&
       \max\Big\{
       \tnorm{\Delta_L+\Delta_L\tp+\Delta_L\Delta_L\tp},
       \tnorm{\Delta_R\tp+\Delta_R+\Delta_R\tp\Delta_R}
       \Big\}, \\
  \beta :=\;&
              \max\Big\{
              \tnorm{(I+\Delta_L)\tp-I},
              \tnorm{(I+\Delta_R)-I}
              \Big\} \\
  =\;&
       \max\{\tnorm{\Delta_L\tp},\tnorm{\Delta_R}\}, \\
  \delta :=\;&
               \max\Big\{
               \tnorm{(I+\Delta_L)(I+\Delta_L)\tp},
               \tnorm{(I+\Delta_R)\tp(I+\Delta_R)}
               \Big\} \times \\
             &\max\Big\{
               \tnorm{\big((I+\Delta_L)\tp(I+\Delta_L)\big)\inv-I},
              \tnorm{\big((I+\Delta_R)(I+\Delta_R)\tp\big)\inv-I}
               \Big\}.
\end{align*}

Let us now seek the opportunity to drop the $-\alpha$ term in the denominator of \eqref{eq.original-result-sin-angle}. We first construct an upper bound for $\alpha$:
\begin{equation}\notag
  \alpha  \le
  \max\{2\tnorm{\Delta_L}+\tnorm{\Delta_L}^2,\;
  2\tnorm{\Delta_R}+\tnorm{\Delta_R}^2\}
  \le
  3\max\{\tnorm{\Delta_L},\tnorm{\Delta_R}\}.
\end{equation}
Applying the assumption,
$3\max\{\tnorm{\Delta_L},\tnorm{\Delta_R}\}\le \svgap(B,k)/2$,
on~\eqref{eq.original-result-sin-angle} yields
\begin{align}
  \notag
  \sin \angle\big(v_k(B), v_k(B+\Delta B)\big)
  & \le \sqrt{2}\bigg(\frac{\delta}{\svgap(B,k)/2}+\beta\bigg) \\
  \notag
  & \le \sqrt{2}\bigg(\frac{2\delta+\beta\,\svgap(B,k)}{\svgap(B,k)}\bigg) \\
  \label{eq.sin-angle-intermediate-proof-2}
  & \le 2\sqrt{2} \bigg(\frac{\delta+\beta}{\svgap(B,k)} \bigg)
\end{align}
The last inequality is due to the definition of $\svgap(B,k)$ implies $\svgap(B,k)\le 2$. Now, it remains to bound $\delta$. First,
\begin{equation}
  \notag
  \max\Big\{
  \tnorm{(I+\Delta_L)(I+\Delta_L)\tp},
  \tnorm{(I+\Delta_R)\tp(I+\Delta_R)}
  \Big\}
  \le
  1+3\max\{\tnorm{\Delta_L},\tnorm{\Delta_R}\}.
\end{equation}
The second component of $\delta$ can be bounded by using the property of the Neumann series~\cite[Chap.~1, Thm.~4.20]{stew98-MA1}. Since $\max\{\tnorm{\Delta_L},\tnorm{\Delta_R}\}<1/6$, we have 
\begin{align*}
  \notag
  & 
  \max 
  \Big\{
   \tnorm{\big((I+\Delta_L)\tp(I+\Delta_L)\big)\inv-I},
  \tnorm{\big((I+\Delta_R)(I+\Delta_R)\tp\big)\inv-I}
   \Big\} \\
   & = \max\Big\{ 
   \tnorm{(I+\Delta_L\tp+\Delta_L +\Delta_L\tp \Delta_L)\inv - I},
   \tnorm{(I+\Delta_R+\Delta_R\tp +\Delta_R\Delta_R\tp)\inv - I}
   \Big\}
   \\
   & \le 
  \max\bigg\{
  \frac{3\tnorm{\Delta_L}}{1-3\tnorm{\Delta_L}},
  \frac{3\tnorm{\Delta_R}}{1-3\tnorm{\Delta_R}}
  \bigg\} \\
  & \le \frac{3\max\{\tnorm{\Delta_L},\tnorm{\Delta_R}\}}{1-3\max\{\tnorm{\Delta_L},\tnorm{\Delta_R}\}}.
\end{align*}
Consequently,
\begin{equation} \notag
  \delta
  \le
  \frac{1+3\max\{\tnorm{\Delta_L},\tnorm{\Delta_R}\}}
  {1-3\max\{\tnorm{\Delta_L},\tnorm{\Delta_R}\}}
  \cdot
  3\max\{\tnorm{\Delta_L},\tnorm{\Delta_R}\}.
\end{equation}
Using again the assumption $\max\{\tnorm{\Delta_L},\tnorm{\Delta_R}\}<1/6$, we obtain
\begin{equation}\notag
  \delta \le 9\max\{\tnorm{\Delta_L},\tnorm{\Delta_R}\}.
\end{equation}
Now, we have $\delta+\beta \le 10\max\{\tnorm{\Delta_{L}}, \tnorm{\Delta_{R}}\}$.
Substituting this into~\eqref{eq.sin-angle-intermediate-proof-2} completes the proof.
\end{proof}
\end{document}